\documentclass[11pt]{amsproc}

\usepackage{a4wide}
\usepackage{amsmath, amssymb, amsthm, latexsym,nicefrac,setspace}

\usepackage[colorlinks=true, citecolor=red, hypertexnames=true, pagebackref]{hyperref}
\usepackage{amsrefs}

\usepackage[all]{xy}

\theoremstyle{plain}
\newtheorem{theorem}{Theorem}[section]
\newtheorem{question}[theorem]{Question}
\newtheorem{lemma}[theorem]{Lemma}
\newtheorem{corollary}[theorem]{Corollary}
\newtheorem{proposition}[theorem]{Proposition}
\newtheorem{definition}[theorem]{Definition}
\newtheorem{conjecture}[theorem]{Conjecture}

\theoremstyle{remark}
\newtheorem{remark}[theorem]{Remark}
\newtheorem*{theorem*}{Theorem}

\def\dim{{\rm dim}}
\def\Z{\mathbb Z}
\def\Q{\mathbb Q}
\def\C{\mathbb{C}}

\def\N{\mathbb{N}}

\def\SU{{\rm SU}}
\def\PU{{\rm PU}}

\def\U{{\rm U}}
\def\F{{\mathbf F}}

\title[Equations over groups]{New topological methods to solve equations over groups}

\author{Anton Klyachko}
\address{A.K., Moscow State University, Russia.}
\email{klyachko@mech.math.msu.su}

\author{Andreas Thom}
\address{A.T., Institut f\"ur Geometrie, TU Dresden, Germany}
\email{andreas.thom@tu-dresden.de}

\begin{document}


\begin{abstract}
We show that the equation associated with a group word $w \in G \ast \F_2$ can be solved over a hyperlinear group $G$ if its content -- that is its augmentation in $\F_2$ -- does not lie in the second term of the lower central series of $\F_2$. Moreover, if $G$ is finite, then a solution can be found in a finite extension of $G$. The method of proof extends techniques developed by Gerstenhaber and Rothaus, and uses computations in $p$-local homotopy theory and cohomology of compact Lie groups.
\end{abstract}

\maketitle

\tableofcontents

\section*{Introduction}

This paper is about the solvability of equations in groups. Let us start by briefly recalling the analogous situation of polynomial equations with rational coefficients. Even though not every non-constant polynomial $p(t) \in \Q[t]$ has a root in $\Q$, there always exists a finite field extension $\Q \subset K$, such that $p(t)=0$ can be solved in $K$, i.e., there exists $\alpha \in K$ with $p(\alpha)=0$. Indeed, it is straightforward to construct some splitting field $K$ with the desired property using the machinery of commutative algebra. On the other side, it is also well-known that arguments from algebraic topology (using notions of degree, winding number or the fundamental group) can be used to show that every polynomial has a root in the topological field $\C$ and historically, this was the first way to provide a field extension of $\Q$ in which $p(t)=0$ can be solved -- an argument which essentially goes back to Gauss' work from 1799.
In the analogous situation, when one wants to solve equations with coefficients in a group, the algebraic or combinatorial approach fails to a large extend and the  homotopy theoretic approach has been used by Gerstenhaber-Rothaus \cite{MR0166296} to obtain positive results -- see the next section for definitions and more precise statements. As a particular consequence, Gerstenhaber-Rothaus were able to prove the Kervaire-Laudenbach conjecture for locally residually finite groups. First of all, we want to clarify a relationship between Connes' Embedding Problem and the Kervaire-Laudenbach Conjecture that was observed in \cite{MR2460675}. 
Moreover, we want to extend this study to cover a larger class of groups and also a larger class of equations which can be handled by methods from algebraic topology. Our methods involve a detailed study of the $p$-local homotopy type of the simple Lie group ${\rm PU}(p)$ and the effect of word maps on the cohomology ring with mod $p$ coefficients.

Our first main result applies to any prime number $p$ and any group word $w \in \SU(p) \ast {\mathbf F}_2$. If the augmentation of $w$ in ${\mathbf F}_2$ does not lie in the second step of the exponent-$p$ central series $[{\mathbf F}_2,{\mathbf F}_2]^p [{\mathbf F}_2,[{\mathbf F}_2,{\mathbf F}_2]],$ then the equation $w(a,b)=1$ can be solved in $\SU(p)$. This implies our second main result, which says that a two-variable equation with augmentation not in $[{\mathbf F}_2,[{\mathbf F}_2,{\mathbf F}_2]]$ can be solved over any hyperlinear group, see Section \ref{prel} for details. Moreover, if the group $G$ is finite, then a solution can be found in a finite extension of $G$. This covers classes of singular equations, which were intractable by combinatorial methods or the topological methods developed by Gerstenhaber-Rothaus.
Our main results are stated explicitly as Theorems \ref{mainthm} and \ref{sup} in Section \ref{sec:stat}.

\vspace{0.2cm}

The paper is organized as follows. Section \ref{prel} collects various preliminaries and discusses briefly the setup of group words and equations and the classes of hyperlinear and sofic groups. Section \ref{topsu} recalls some facts about the cohomology of $\SU(n)$ and $\PU(n)$, localization theory of topological spaces, and computations of homotopy groups of spheres. This section is the most technical part and also contains a review and extension of results of Kishimoto and Kono \cite{MR2544124}. Section \ref{solv} contains the proofs of our main results and discusses related low-dimensional results and further directions.

\section{The main results} \label{prel}

\subsection{Group words and equations}
We denote by ${\mathbf F}_n$ the free group on generators $x_1,\dots,x_n$. For any group $G$, an element $w$ in the free product $G \ast {\mathbf F}_n$  determines a word map $w \colon  G^{\times n} \to G$ given by evaluation. We denote by $\varepsilon \colon  G \ast {\mathbf F}_n \to {\mathbf F}_n$ the natural augmentation which sends $G$ to the neutral element and call $\varepsilon(w)$ the {\it content} of $w$. We call $w \in G \ast {\mathbf F}_n$ a group word in $n$ variables with coefficients in $G$. Every group word $w \in G \ast {\mathbf F}_n$ determines an equation $w(x_1,\dots,x_n)=1$ in $n$ variables with coefficients in $G$ in an obvious way. We say that $w \in G \ast {\mathbf F}_n$ can be solved {\it over} $G$ if there exists an overgroup $H \supseteq G $ and $g_1,\dots,g_n \in H$ such that $w(g_1,\dots,g_n)=1$, where $1$ denotes the neutral element in $H$. Similarly, we say that it can be solved {\it in} $G$ if we can take $H=G$. We denote the normal subgroup in $G \ast \F_n$, which is generated by the element $w$ by $\langle\!\langle w \rangle\!\rangle$. It is clear that an equation $w \in G \ast {\mathbf F}_n$ can be solved over $G$ if and only if the natural homomorphism $G \to G \ast {\mathbf F}_n/\langle\!\langle w \rangle\! \rangle$ is injective. Similarly, an equation can be solved in $G$ if and only if the natural homomorphism $G \to G \ast {\mathbf F}_n/\langle\!\langle w \rangle\! \rangle$ is split-injective, i.e., it has a left inverse. 

The study of equations over groups dates back to the work of Bernhard Neumann \cite{MR0008808}. There is an extensive literature about equations over groups, including \cites{MR2975932, MR2785779, MR2642015, MR1002920, MR2100369, MR919828, MR0166296, MR614523, MR695646, MR1772011, MR1218513, MR1696317, MR2179667, MR0142643, MR0008808}. See also Roman'kov's recent survey about this topic \cite{MR3043434}.

\vspace{0.2cm}

It is well-known that not all equations with coefficients in $G$ are solvable over $G$. For example if $G=\langle a,b \mid a^2,b^3 \rangle$, then the equation $w(x) = xax^{-1}b$ with variable $x$ is not solvable over $G$. Indeed, $a$ and $b$ cannot become conjugate in any overgroup of $G$. Another example is $G = \Z/p\Z = \langle a \rangle$ with the equation $w(x) = xax^{-1} a xa^{-1}x^{-1}a^{-2}$. 
However, in both cases we have $\varepsilon(w)=1 \in {\mathbf F}_n$. Indeed, the only known examples of equations which are not solvable over some $G$ are equations whose content is trivial. We call an equation $w \in G \ast {\mathbf F}_n$ singular if its content is trivial, and non-singular otherwise. This lets us put forward the following conjecture:

\begin{conjecture} \label{conj}
Let $G$ be a group and $w \in G \ast {\mathbf F}_n$ be an equation in $n$ variables with coefficients in $G$. If $w$ is non-singular, then it is solvable over $G$. In addition, if $G$ is finite, then a solution can be found in a finite extension.
\end{conjecture}

The case $n=1$ is the famous Kervaire-Laudenbach Conjecture.
The one-variable case was studied in classical work by Gerstenhaber-Rothaus from 1962, see \cite{MR0166296}. They showed that if $G$ is finite, then every non-singular equation in one variable can be solved over $G$ (in fact in some finite extension of $G$). Their proof used computations in cohomology of the compact Lie groups $U(n)$. It is this proof, that inspired us to start this work. The work of Gerstenhaber-Rothaus showed in fact that every non-singular equation in one variable with coefficients in the unitary group ${\rm U}(n)$ can be solved already {\it in} ${\rm U}(n)$, for any $n \in \N$. Their strategy was to use homotopy theory to say that the associated word map $w \colon {\rm U}(n) \to {\rm U}(n)$ has a non-vanishing degree (as a map of oriented manifolds) and thus must be surjective. Any preimage of the neutral element provides a solution to the equation $w$. The key to the computation of the degree is to observe that the degree depends only on the homotopy class of $w$ and thus -- since ${\rm U}(n)$ is connected -- does not change if $w$ is replaced by $\varepsilon(w)$. The computation of the degree is now an easy consequence of classical computations of Hopf \cite{hopf}.

This property of solvability {\it in} a group is easily seen to pass to arbitrary Cartesian products of groups and arbitrary quotients of groups. As a consequence, non-singular equations in one variable with coefficients in $G$ as above can be solved over $G$ if $G$ is isomorphic to a subgroup of a quotient of the infinite product $\prod_n \! {\rm U}(n)$ -- an observation that is due to Pestov \cite{MR2460675}. Groups which admit such an embedding are called hyperlinear groups, see \cite{MR2460675} and Section \ref{hyperlinear} for more information on this class of groups -- see also Remark \ref{hyperl}. Thus, the result of Gerstenhaber-Rothaus also holds for hyperlinear groups -- in particular for all amenable groups, or more generally, all sofic groups \cite{MR2460675}. Connes' Embedding Conjecture predicts (among other things) that every countable group is hyperlinear and thus implies the Kervaire-Laudenbach Conjecture -- this was also observed by Pestov in \cite{MR2460675}.

\vspace{0.2cm} Actually, Gerstenhaber-Rothaus \cite{MR0166296} studied the more involved question whether $m$ equations of the form $w_1,\dots,w_m \in G \ast {\mathbf F}_n $ in $n$ variables can be solved simultaneously over $G$. Their main result is that this is the case if $G$ is finite (or more generally, locally residually finite) and the presentation two-complex $X := K\langle x_1,\dots,x_n \mid \varepsilon(w_1),\dots,\varepsilon(w_m) \rangle$ satisfies $H_2(X,\mathbb Z)=0$, i.e., the second homology of $X$ with integral coefficients vanishes. Here, the representation two-complex is the two-dimensional $CW$-complex associated with the presentation of a group -- obtained by glueing $m$ two-cells to a bouquet of $n$ circles according to the relations. (This amounts to a certain algebraic condition on the exponent sum matrix.) A system of equations which satisfies this vanishing condition was called non-singular by Jim Howie \cite{MR614523} -- note that our terminology is not consistent with this, but there will be no risk of confusion.
Later, Howie \cite{MR614523} proved the same result for locally indicable groups and conjectured it to hold for all groups -- we call that Howie's Conjecture. Again, Connes' Embedding Conjecture implies Howie's Conjecture -- and more specifically, every hyperlinear group satisfies Howie's Conjecture. 

\begin{remark}
Equations in one variable with at most three occurrences of the variable are solvable by a result of Howie \cite{MR695646}, which however also reduces this to the residually finite case and uses the results of Gerstenhaber-Rothaus. Similar results have been proved for non-singular equations with four \cite{MR1002920} and five \cite{MR2324624} occurrences of variables.
\end{remark}

\begin{remark}  Equations $w \in G \ast \F_n$ with $G$ torsion-free can be solved by more combinatorial methods. A systematic study of the torsion-free case was started by Levin \cite{MR0142643} who conjectured that equations in one variable with coefficients in a torsion-free group should always be solvable if $w$ is conjugate to an element in $G$. A result in this direction is due to the first author who proved that this is indeed the case for one-variable equations with content $\pm1 \in \Z$, see \cite{MR1218513}. Moreover, over any torsion-free group any several-variable equation whose content is not a proper power (and not the neutral element) is solvable \cite{MR2251364}.
Again, due to the absence of any counterexamples, it is conjectured that solvability in the torsion-free case is true even if the content of the equation is trivial.
\end{remark}

\begin{remark} Existence of non-trivial solutions can be a subtle issue too, see for example \cite{rothaus}.
For $G= \Z/p\Z = \langle a \rangle$, the equation $w(x) = axa^{-1}xax^{-1}a^{-1}x^{-2}$ can be solved in a finite overgroup of $G$ only with $x=1$, even though non-trivial solutions exist in infinite extensions. The mechanism behind these kind of examples was first discovered by Higman, see \cite{higman}.
\end{remark}

\subsection{Statement of the main results} \label{sec:stat}
The main goal of this work is to provide examples of non-singular equations in many variables which are solvable over every hyperlinear group, where the condition on the equation {\it only} depends on its content. This should be compared for example with results of Gersten \cite{MR919828}, where the conditions on $w$ depended on the unreduced word obtained by deleting the coefficients from $w$. For simplicity, we concentrate on the two-variable case. Our main result is:

\begin{theorem} \label{mainthm}
Let $G$ be a hyperlinear group. An equation in two variables with coefficients in $G$ can be solved over $G$ if its content does not lie in $[\F_2,[\F_2,\F_2]]$. Moreover, if $G$ is finite, then a solution can be found in a finite extension of $G$.
\end{theorem}

In order to prove our main result we have to refine the study of Gerstenhaber-Rothaus on the effect of word maps on cohomology of compact Lie groups. 
Again, the strategy is to show that such equations can be solved in $\SU(n)$ for sufficiently many $n \in \N$. More specifically, we prove:

\begin{theorem} \label{sup}
Let $p$ be a prime number. Let $w \in \SU(p) \ast {\mathbf F}_2$ be a group word. If $$\varepsilon(w) \not \in [{\mathbf F}_2,{\mathbf F}_2]^p [{\mathbf F}_2,[{\mathbf F}_2,{\mathbf F}_2]],$$ then the equation $w(a,b)=1$ can be solved in $\SU(p)$.
\end{theorem}

If $\varepsilon(w) \not \in [\F_2,\F_2]$, then this theorem is a direct consequence of the work of Gerstenhaber and Rothaus. However, if $\varepsilon(w) \in [\F_2,\F_2]$, then a new idea is needed. We show -- under the conditions on $p$ which are mentioned above -- that the induced word map $w \colon  \PU(p) \times \PU(p) \to \SU(p)$ is surjective, where $\SU(p)$ denotes the special unitary group and $\PU(p)$ its quotient by the center. The strategy is to replace $w$ by the much simpler and homotopic map induced by $\varepsilon(w)$ and study its effect on cohomology directly. This is done in Section \ref{sec2} with the necessary preparations from Section \ref{topsu}. 

In general, the assumption on $\varepsilon(w)$ cannot be omitted in the previous theorem. Indeed, the second author showed in previous work:

\begin{theorem}[\cite{MR3043070}]
For every $k \in \N$ and every neighborhood ${\mathcal V} \subset \SU(k)$ of $1_k \in \SU(k)$, there exists $w \in {\mathbf F}_2 \setminus \{e\}$, such that the image of $w \colon   \SU(k)^{\times 2} \to  \SU(k)$ lies in ${\mathcal V}$. In particular, the equation $w(a,b)=g$ is not solvable in $\SU(k)$ for $g \not\in {\mathcal V}$.
\end{theorem}

The construction that proves the preceding theorem yields words in $\F_2$ that lie deep in the derived series, so that there is no contradiction with Theorem \ref{sup}. 

The surjectivity of word maps without coefficients is an interesting subject in itself. Michael Larsen conjectured that for each non-trivial $w \in \F_2$ and $n$ high enough, the associated word map $w \colon  \PU(n) \times \PU(n) \to \PU(n)$ is surjective. This was shown (with some divisibility restrictions on $n$) for words not in the second derived subgroup of $\F_2$ by Elkasapy and the second author in \cite{MR3043070}.
In a similar direction, we believe that for $n$ high enough -- or again, with some divisibility restrictions -- the word map $w$ should define a non-trivial homotopy class and not even be homotopic to a non-surjective map.

\vspace{0.2cm}

In order to study words which lie deeper in the lower central series, we suspect that it might be helpful to oberserve that the induced word map $w \colon  \PU(p) \times \PU(p) \to \PU(p)$ does not only lift to $\SU(p)$  -- which is the simply connected cover of $\PU(p)$ -- but lifts even to higher connected covers of $\PU(p)$. Indeed, for example one can show that if $w \in [\F_2,[\F_2,\F_2]]$, then the associated word maps lifts to the complex analogue of the string group, see \cite{MR2079378} for a study of related groups.

\section{The topology of $\SU(p)$ and $\PU(p)$}
\label{topsu}
In this section we collect some standard results from algebraic topology that will be used in the proof of the main theorems.
A classical result of Samelson says:

\begin{theorem}[Samelson, see \cite{MR0065157}]
The commutator map $c_2^{{\rm \SU(2)}} \colon  {\rm SU}(2)^{\times 2} \to {\rm SU}(2)$ is not null-homotopic. In particular,  since ${\rm SU}(2)$ is a sphere, any map homotopic to the commutator map must be surjective.
\end{theorem}

This easily implies that for every group word $w \in {\rm SU}(2) \ast {\mathbf F}_2$ whose content is the commutator of the generators of ${\mathbf F}_2$, the equation $w(a,b)=1$ can be solved in ${\rm SU}(2)$. This already has non-trivial consequences that (to the best of our knowledge) could not be proved using combinatorial techniques. In order to treat ${\rm SU}(n)$ for higher $n$, we have to recall some aspects of algebraic topology.
Our methods in the proof of the main results follow closely ideas from Hamanaka-Kishimoto-Kono \cite{MR2243721} and Kishimoto-Kono \cite{MR2544124}.

\subsection{The cohomology of $\SU(p)$ and $\PU(p)$}
\label{cohom}
 Let $n$ be a positive integer. We denote by $\SU(n)$ the special unitary group and by $\PU(n)$ the quotient of $\SU(n)$ by its center -- the projective unitary group. We denote the quotient map by $\pi \colon \SU(n) \to \PU(n)$ and the coset of some $u \in \SU(n)$ in $\PU(n)$ by $\pi(u)=\bar u$.  
The cohomology rings of the simply-connected classical Lie groups were computed by Borel in \cite{MR0051508}. For example, it is well-known that as a graded ring $$H^*(\SU(n),\Z/p\Z) = \Lambda_{\Z/p\Z}^*(x_2,x_3,\dots,x_{n})$$ with $|x_i| = 2i-1$. Here, we denote by $\Lambda^*_k$ the exterior algebra over a field $k$ on a certain set of generators of particular degrees. The product map $m \colon \SU(n) \times \SU(n) \to \SU(n)$ and the inversion turn $H^*(\SU(n),\Z/p\Z)$ into a Hopf algebra. However, the comultiplication turns out to be trivial in this situation, i.e., $\Delta(x_i)= x_i \otimes 1 + 1 \otimes x_i$ for all $2 \leq i \leq n$. We will be mainly interested in the case $n=p$.
The computation of the cohomology ring of $\PU(p)$ is more involved than that of $\SU(p)$ and was also first studied by Borel in \cite{MR0064056}. Later, the comultiplication on the cohomology of $\PU(p)$ with $\Z/p\Z$-coefficients was computed in work of Baum-Browder \cite{MR0189063} and turns out to be not co-commutative. It is this lack of co-commutativity which makes our approach work. Let us sumarize the situation:

\begin{theorem}[Baum-Browder, see \cite{MR0189063}] 
\label{baumbrowder}
Let $p$ be an odd prime number. Then,
$$H^*(\PU(p),\Z/p\Z) \cong (\Z/p\Z)[y]/(y^{p}) \otimes_{\Z} \Lambda_{\Z/p\Z}^*(y_1,y_2,\dots,y_{p-1})$$ with $|y|=2$, $|y_i| = 2i-1$, $\pi^*(y_i)=x_{i}$ for $2 \leq i \leq p-1$, and $\pi^*(y)=\pi^*(y_1)=0$. Moreover, the comultiplication takes the form
$$\Delta(y)=y\otimes 1 + 1 \otimes y, \quad \mbox{and} \quad\Delta(y_i)=y_i \otimes 1 + 1 \otimes y_i + \sum_{j=1}^{i-1} \binom{j-1}{i-1} \cdot y_{j}\otimes y^{i-j}.$$
\end{theorem}

We denote by $\mathcal I(n)$ the kernel of the natural augmentation $H^*(\PU(p)^n,\Z/p\Z) \to \Z/p\Z$. We start out by recalling the effect of various natural word maps on the cohomology ring.

\begin{lemma} \label{coprod}
Let $n \in \N$ and $p$ be an odd prime number. Consider the  map $\mu_n \colon  \PU(p) \to \PU(p)$ given by $\mu_n(u)=u^n$.
The map induced by $\mu_n \colon  \PU(p) \to \PU(p)$ on cohomology satisfies
$$\mu_n^*(y_i) = n \cdot y_i \mod {\mathcal I}(1)^2$$
for all $1 \leq i \leq p-1$.
\end{lemma}
\begin{proof}
The map $\mu_n$ arises as the composition of the diagonal embedding ${\rm PU}(p) \to \prod_{i=1}^n {\rm PU}(p)$ following by the multiplication map $m \colon \prod_{i=1}^n {\rm PU}(p) \to {\rm PU}(p)$. In cohomology, this induces first the the $n$-fold coproduct
$\Delta^{(n)} \colon H^*(\PU(p),\Z/p\Z) \to \otimes_{i=1}^n H^*(\PU(p),\Z/p\Z)$ followed by the multiplication in cohomology $\otimes_{i=1}^n H^*(\PU(p),\Z/p\Z) \to H^*(\PU(p),\Z/p\Z)$.
An easy verification shows that each generator is just multiplied by $n$, modulo sums of products of at least two generators. This proves the claim.
\end{proof}

It is also important for us to study the effect of the commutator map in cohomology. We need the following result from the work of Hamanaka-Kishimoto-Kono \cite{MR2243721}, which is an easy consequence of Theorem \ref{baumbrowder} above.

\begin{lemma}[Proposition 6 in \cite{MR2243721}] \label{comm}
The commutator map
$$c \colon  \PU(p)\times\PU(p)\to\PU(p)$$ induces the cohomology map
$c^*\colon H^*(\PU(p),\Z/p\Z)\to H^*(\PU(p)\times\PU(p),\Z/p\Z)$ sending $y_i$ 
to $(i-1)(y_{i-1}\otimes y-y\otimes y_{i-1})$ modulo the ideal $\mathcal I(2)^3$. The elements $y_1$ and $y$ map to zero.
\end{lemma}

Note that the commutator map $\SU(n) \times \SU(n) \ni (u,v) \mapsto uvu^{-1}v^{-1} \in \SU(n)$ induces a well-defined map $c \colon  \PU(n) \times \PU(n) \to \SU(n)$, which we will also call commutator map.

Our first aim is to show that the commutator map $c \colon  \PU(p) \times \PU(p) \to \SU(p)$ is not homotopic to a non-surjective map. We will show this by showing that the image of the top-dimensional cohomology class $x_2\cdots x_p \in H^{p^2-1}(\SU(p),\Z/p\Z)$ does not vanish in the group $H^{p^2-1}(\PU(p) \times \PU(p),\Z/p\Z).$ As it turns out, the study the images of $x_2,\dots,x_{p-1}$ in $H^{*}(\PU(p) \times \PU(p),\Z/p\Z)$ is fairly straightforward, since these generators are the image of generators $y_2,\dots,y_{p-1}$ in the cohomology of $\PU(p)$. The study of the last generator $x_p$ is considerably more complicated and we have to rely on some structure results on the $p$-local homotopy type of ${\rm PU}(p)$. In fact, in the final argument we will not rely on Lemma \ref{comm}, but the required result is proved for all $x_i$ with $2 \leq i \leq p$ directly.

\subsection{The lens spaces}
\label{lens}
Since $S^{2d-1} \subset \C^p$, there is a natural $(\Z/p\Z)$-action on $S^{2d-1}$ given by scalar multiplication with the complex number $\exp(2 \pi i/p)$. In our considerations, we only study the case $d=p$. We denote by $L$ the lens space $S^{2p-1}/(\Z/p\Z)$ and let $\rho \colon  S^{2p-1} \to L$ be the natural projection. The space $L$ has a natural CW-structure with one cell in each dimension, see \cite[Example 2.43]{MR1867354}.

We denote by $L_k := S^{2k-1} \cup_p D^{2k}$ the mod-$p$ Moore space for $1 \leq k \leq p-1$. This is, by definition, the space obtained by attaching $D^{2k}$ to $S^{2k-1}$ along the attaching map $p \colon S^{2k-1} \to S^{2k-1}$, which is defined as $(z_1,\dots,z_k) \to (z_1^p,z_2,\dots,z_k)$. Note that the letter $p$ is over-used here, but this will not cause any confusion. The characteristic property of $L_k$ (for $1 \leq k \leq p-1$) is that $H_n(L_k,\Z)= 0$ unless $n=2k-1$ or $n=0$, and $H_{2k-1}(L_k,\Z)=\Z/p\Z$. See \cite[Chapter 2, Example 2.40]{MR1867354} for more details. We set $L_p := S^{2p-1}$. 

For $1 \leq k \leq p-1$, we denote by $q_k \colon  L_k \to S^{2k}$ the so-called pinch map, which collapses $S^{2k-1}$ (and hence the boundary of $D^{2k}$) to a point. Note that the 2-skeleton of $L$ is just $L_1$. Indeed, $L_1 = S^1 \cup_p D^2$, see \cite[Example 2.43]{MR1867354}.

There is a fibre bundle
$\SU(p-1)\mathop\to\limits^\iota\SU(p)\mathop\to\limits^\pi S^{2p-1},$
where the embedding $\iota$ sends a matrix $A$ to the matrix
$\left(\begin{smallmatrix} 1&0 \\0&A  \end{smallmatrix} \right)$ and the projection $\pi$ sends a matrix to its first
row.
Similarly, the group $\PU(p)$ admits a
fibre bundle
$\SU(p-1) \stackrel{\iota}{\to}\PU(p) \stackrel{\pi}{\to} L,$
where the embedding $\iota$ sends a matrix $A$ to the class of matrices
$\langle \exp(2\pi i/p)\rangle \left(\begin{smallmatrix} 1&0 \\0&A \end{smallmatrix} \right),$ the projection $\pi$ sends a matrix
to its first row (which is defined up to multiplication by $\exp(2\pi i/p)$),
and $L=S^{2p-1}/ \langle \exp(2\pi i/p) \rangle$ is the lens space.

\subsection{Localization at a prime}
 We will freely use the concept of localization of topological spaces (simply connected or with abelian fundamental group) at a prime $p$, see the work of Bousfield-Kan \cite{MR0365573} or Mimura-Nishida-Toda \cite{MR0295347} background and as general references. See also \cite{MR2884233} for a more recent presentation of this material.

Given a topological space $X$ with abelian fundamental group, we denote by $X_{(p)}$ its $p$-localization which comes equipped with a natural map $\iota \colon  X \to X_{(p)}$.
The $p$-localization can be defined as a certain tower of spaces, and its defining properties are
$$\pi_i(X_{(p)},\iota(x)) = \pi_i(X,x) \otimes_{\Z} \Z_{(p)}, \quad \forall x \in X, \forall i \geq 1.$$
Here, $\Z_{(p)}$ denotes the $p$-localization of $\Z$, i.e., the ring of those fractions in $\Q$, whose denominator is not divisible by $p$.
For a continuous map $f \colon  X \to Y$ (between topological spaces with abelian fundamental group), we denote by $f_{(p)} \colon  X_{(p)} \to Y_{(p)}$ the induced map between the $p$-localizations. We will freely use that $\iota \colon X \to X_{(p)}$ induces an isomorphism on cohomology with coefficients in $\Z/p\Z$. A map is called a $p$-local homotopy equivalence if it is a homotopy equivalence after $p$-localization.

If $X$ is a double suspension, then $[X,Y]$ is an abelian group -- here $[X,Y]$ denote as usual the set of homotopy classes of maps from $X$ to $Y$. We will use that if in addition $X$ is also a finite CW-complex, then 
the natural maps
$$[X,Y] \otimes_{\Z} \Z_{(p)} \to [X,Y_{(p)}] \leftarrow [X_{(p)},Y_{(p)}]$$
are isomorphisms, see for example \cite[Chapter 6.6]{MR2884233}.

We will need the following computation of the homotopy groups of $\SU(n)$, which is due to Bott \cite{MR0102803}.

\begin{theorem}[Bott]
$$\pi_{k}(\SU(n)) = \begin{cases} 0 & k=0,1,2 \\
\Z & k=2i-1, \ 2 \leq i \leq n \\
0& k=2i, \ 2 \leq i \leq n-1 \\
\Z/n!\Z & k=2n 
\end{cases}.$$
\end{theorem}

Let $\epsilon_k \colon  S^{2k-1} \to \SU(n)$ be a generator of $\pi_{2k-1}(\SU(n))$ for $2 \leq k \leq n$ and consider the map $\mu \colon  \prod_{k=2}^n S^{2k-1} \to \SU(n)$ given by $
\mu(x_2,\dots,x_n) := \epsilon_2(x_2) \cdots \epsilon_n(x_n),$ where we use multiplication in the group $\SU(n)$. The following theorem was first proved by Serre \cite[Proposition 7]{MR0059548} -- even though without using the language of localization at the level of topological spaces.

\begin{theorem}[Serre] \label{serre} Let $p$ be a prime number. If $p \geq n$, then the map $$\mu \colon  \prod_{k=2}^n S^{2k-1} \to \SU(n)$$ is a $p$-local homotopy equivalence.
\end{theorem}

We will now concentrate on the case that $n=p$.
For $2 \leq i \leq p$, we denote by $\lambda_i \colon  \SU(p)_{(p)} \to S^{2i-1}_{(p)}$ the composition of the homotopy inverse of $\mu_{(p)}$ with the projection onto $S^{2i-1}_{(p)}$.
We also have the following computation of the $p$-local homotopy groups of odd spheres, which is also due to Serre  \cite[Proposition 11]{MR0059548}:
\begin{equation}
\label{homserre}
\pi_{k}\left(S^{2i-1}_{(p)}\right) = \begin{cases} 
0 & 0 \leq k < 2i-1 \\
\Z_{(p)} & k=2i-1 \\
0 & 2i-1<k< 2i+2p-4\\
\Z/p\Z & k= 2i+2p-4\\
0 & 2i+2p-4 < k< 2i +4p -7.
\end{cases}
\end{equation}
Here, the generator in $\pi_{2i+2p-4}(S^{2i-1}_{(p)})$ is equal to  $\Sigma^{2i-4}(\alpha)$ for some generator $\alpha \in \pi_{2p}(S^3_{(p)})$. Here, $\Sigma(?)$ denotes as usual the suspension also on the level of maps. Hence, Theorem \ref{serre} together with the computation above implies that there is a more refined computation of the homotopy groups of $\SU(p)$ localized at a prime $p$:

$$\pi_k\left(\SU(p)_{(p)}\right) =\begin{cases} 0 & k=0,1,2 \\
\Z_{(p)} & k=2i-1, \ 2 \leq i \leq p \\
0 & k=2i, \ 2 \leq i \leq p-1 \\
0 & k=2i+1, \ p \leq i < 2p-1 \\
\Z/p\Z & k=2i,\ p \leq i < 2p-1. \end{cases}$$

Note that this covers all dimensions up to $4p-4$, whereas Bott's computation only gives information up to dimension $2p$ -- a fact that will be used later.

Now, the commutator map on $\SU(p)$ induces a secondary operation $$\langle.,.\rangle \colon  \pi_i(\SU(p)) \times \pi_j(\SU(p)) \to \pi_{i+j}(\SU(p)),$$ the so-called Samelson product, which was originally introduced in \cite{MR0065157}. Bott already analyzed the Samelson products of the maps $\epsilon_k$ in \cite{MR0123330}. He proved as a corollary to his main result \cite[Theorem 1]{MR0123330} that the element $\langle \epsilon_i, \epsilon_{p-i+1} \rangle$ in $\pi_{2p}(\SU(p)) = \Z/p!\Z$ is divisible by precisely $\frac{p!}{(i-1)!(p-i)!}$, i.e., it is equal to $(i-1)!(p-i)!$ times some generator of $\Z/p!\Z$.
The maps $\epsilon_k$ induce natural maps $\bar \epsilon_k \colon  S^{2k-1} \to \PU(p)$ for $2 \leq k \leq p$. Note, that we can also choose a natural map $\bar \epsilon_1 \colon  S^1 \to \PU(p)$ which yields a generator of $\pi_1(\PU(p))=\Z/p\Z$ -- and that Bott's result extends to the case $i=1$. In the light of our computation of $\pi_k(\SU(p)_{(p)})$ from above, Bott's computation of the Samuelson products implies:

\begin{theorem}[Bott] \label{bott}
Let $p$ be a prime number and $1 \leq i < p$.
The element $$\langle \bar\epsilon_p,\bar \epsilon_{i} \rangle \in \pi_{2p+2i-2}(\SU(p)_{(p)}) = \Z/p\Z$$ does not vanish.
\end{theorem}
\begin{proof}
Indeed, Bott's result from above yields that the image of the map $$\langle \bar\epsilon_p, \bar\epsilon_i \rangle \colon  S^{2p+2i-2} \to \SU(p+i-1)$$ in $\pi_{2p+2i-2}(\SU(p+i-1)) = \Z/(p+i-1)!\Z$ is $(p-1)!(i-1)!$ times some generator, and hence does not vanish modulo $p$. Since $\langle \bar\epsilon_p, \bar\epsilon_i \rangle$ factors through $\SU(p)$, the assertion follows.
\end{proof}

The non-vanishing of these Samuelson products modulo $p$ will be the key to understand the non-vanishing of certain cohomology classes after application of the commutator map.

\subsection{The work of Kishimoto-Kono}

In order to understand the effect of the commutator map on the cohomology of ${\rm PU}(p)$, we must now study the $p$-local homotopy type of ${\rm PU}(p)$. We restate Proposition 2 from the work of Kishimoto-Kono \cite{MR2544124}.

\begin{lemma}
There exists a natural map $\eta \colon  L_{(p)} \to \PU(p)_{(p)}$, such that the diagram
\begin{equation} \label{factorization}
\xymatrix{ S^{2p-1}_{(p)} \ar[r]^{\epsilon_{p(p)}} \ar[d]^{\rho_{(p)}} & \SU(p)_{(p)} \ar[d] \\
L_{(p)} \ar[r]^{\eta} & \PU(p)_{(p)}
}
\end{equation}
commutes up to homotopy.
\end{lemma}

Using the notation introduced in Section \ref{lens}, we are now ready to state and prove an extension of Lemma 4 of Kishimoto-Kono \cite{MR2544124}.

\begin{lemma} \label{lem:crucial}
For $1 \leq i \leq p-1$, we have $p$-locally
$$\lambda_{i+1} \circ c_{(p)} \circ  (\eta|_{L_{1(p)}} \wedge \bar\varepsilon_{i(p)})= a \cdot(q_1 \wedge 1_{S^{2i-1}})_{(p)} \colon  L_{1(p)} \wedge S_{(p)}^{2i-1} \to S_{(p)}^{2i+1}$$
for some $a \in \Z^{\times}_{(p)}$.
\end{lemma}
\begin{proof}
First of all, we know from \cite[Proposition 9.6]{MR0295347} that there exists a $p$-local splitting
$$(L \wedge S^1)_{(p)} = \bigvee_{k=1}^p (L_k \wedge S^1)_{(p)}.$$ Thus, $(L \wedge S^{2i-1})_{(p)} = \bigvee_{k=1}^p (L_k \wedge S^{2i-1})_{(p)}$ and the map $$\rho \wedge 1_{S^{2i-1}} \colon   \underbrace{S^{2p-1} \wedge S^{2i-1}}_{{S^{2p+2i-2}}} \to L \wedge S^{2i-1}$$ can be decomposed $p$-locally as
\begin{equation} \label{eq1}
(\rho \wedge 1_{S^{2i-1}})_{(p)} = \bigvee_{k=1}^p f_k, \quad \mbox{with} \quad f_k \colon  S^{2p+2i-2}_{(p)} \to (L_k \wedge S^{2i-1})_{(p)}, \quad 1 \leq k \leq p.
\end{equation}
Since the map $\rho \colon S^{2p-1} \to L$ is a $p$-fold covering, we get $f_p$ is equal to multiplication by $p$. Now, for $1 \leq k \leq p-1$, the cofiber sequence $S^{2k-1} \stackrel{p}{\to} S^{2k-1} \to L_k$ (coming from the definition of $L_k$) induces a long exact sequence
$$\dots \to \pi_{2i+2k-1}(S^{2i+1}) \stackrel{p}{\to}\pi_{2i+2k-1}(S^{2i+1}) \to [L_k \wedge S^{2i-1}, S^{2i+1}] \to  \pi_{2i+2k-2}(S^{2i+1}) \stackrel{p}{\to} \cdots.$$
From the computations of the $p$-local homotopy groups of spheres, we obtain for all $1 \leq i \leq p$ that
$$[L_k \wedge S^{2i-1}, S^{2i+1}]_{(p)} = \begin{cases} \Z_{(p)} & k=1 \\
0 & 2 \leq k \leq p-1 \\
\Z/p\Z & k=p. \end{cases}$$ 
Moreover, it follows that the group $[L_k \wedge S^{2i-1}, S^{2i+1}]_{(p)}$ is generated by the map $(q_1 \wedge 1_{S^{2i-1}})_{(p)}$.
Consider now the map 
$$\lambda_{i+1} \circ c_{(p)} \circ (\eta \wedge \bar \epsilon_{i(p)}) \colon  L_{(p)} \wedge S^{2i-1}_{(p)} \to S^{2i+1}_{(p)}.$$ We obtain from Equation \eqref{eq1} and the sentence after Equation \eqref{homserre} that
\begin{equation} \label{eqimp}
\lambda_{i+1} \circ c_{(p)} \circ (\eta \wedge \bar \epsilon_{i(p)}) 
= a_i \cdot (q_1 \wedge 1_{S^{2i-1}})_{(p)} \vee b_i \cdot (\Sigma^{2i-2}(\alpha))
\end{equation}
 for some $a_i,b_i \in \Z_{(p)}$ and $1 \leq i \leq p-1$. Here, we consider $\Sigma^{2i-2}(\alpha)$ as map $$L_{p(p)} \wedge S^{2i-1}_{(p)}=S^{2p + 2i-2}_{(p)} = S^{2i-2}_{(p)}\wedge S^{2p}_{(p)} \stackrel{1 \wedge \alpha_{(p)}}{\to} S^{2i-2}_{(p)}\wedge S^3_{(p)} = S^{2i+1}_{(p)}.$$ Now is the point, when we are going to use Theorem \ref{bott}. Indeed, we have the following identification of homotopy classes of maps from $S^{2p+2i-2}$ to $S^{2i+1}$ (note that $\pi_{2p+2i-2}(S^{2i+1})=\Z/p\Z$ by Equation \ref{homserre}):
\begin{eqnarray*}
0 &\stackrel{{\rm Th.} \ref{bott}}{\neq} & \lambda_{i+1} \circ \langle \bar \epsilon_p, \bar \epsilon_i \rangle_{(p)} \\
&=& \lambda_{i+1} \circ c_{(p)} \circ (\bar\epsilon_{p(p)} \wedge \bar\epsilon_{i(p)})  \\
&\stackrel{\eqref{factorization}}=& \lambda_{i+1} \circ c_{(p)} \circ (\eta \wedge \bar\epsilon_{i(p)}) \circ (\rho_{(p)} \wedge 1_{S^{2i-1}})\\
&\stackrel{\eqref{eqimp}}{=}& \left(a_i \cdot (q_1 \wedge 1_{S^{2i-1}})_{(p)} \vee b_i \cdot (\Sigma^{2i-2}(\alpha)) \right) \circ (\rho_{(p)} \wedge 1_{S^{2i-1}}) \\
&\stackrel{\eqref{eq1}}=& a_i \cdot ((q_1 \wedge 1_{S^{2i-1}})_{(p)} \circ f_1) \vee p b_i \cdot (\Sigma^{2i-2}(\alpha)) \\
&=& a_i \cdot ((q_1 \wedge 1_{S^{2i-1}})_{(p)} \circ f_1).\\
\end{eqnarray*}
Finally, this implies that $(q_1 \wedge 1_{S^{2i-1}})_{(p)} \circ f_1$ is a non-zero multiple of the class $$\Sigma^{2i-2}(\alpha) \in \pi_{2p+2i-2}(S^{2i+1})_{(p)} = \Z/p\Z$$ and we can conclude that $a_i \in \Z_{(p)}^{\times}$. This finishes the proof.
\end{proof}

The preceding proof followed closely the work of Kishimoto-Kono \cite{MR2544124} and we do not claim any originality for this computation.

\subsection{Applications to cohomology}

We will now apply the computation from the previous section to study the effect of the commutator map on cohomology. Recall that
$$H^*(\PU(p),\Z/p\Z) \cong (\Z/p\Z)[y]/(y^{p}) \otimes_{\Z} \Lambda_{\Z/p\Z}^*(y_1,y_2,\dots,y_{p-1})$$
with $|y|=2$, $|y_i| = 2i-1$.
We
denote by $\mathcal J_i$ the ideal in the ring $H^*(\PU(p),\Z/p\Z) \otimes H^*(\PU(p),\Z/p\Z)$ which is generated by $y^2 \otimes 1,y_j \otimes 1, 1 \otimes y_k, 1 \otimes y$ for $1 \leq j,k \leq p-1$ with $k \neq i$.

\begin{corollary} \label{maincor} Let $2 \leq i \leq p$ and $c \colon \PU(p) \times \PU(p) \to \SU(p)$ the commutator map.
Then, the induced map $$c^* \colon  H^*(\SU(p),\Z/p\Z) \to H^*(\PU(p),\Z/p\Z) \otimes_{\Z} H^*(\PU(p),\Z/p\Z)$$
satisfies $$c^*(x_i) = a_i(y \otimes y_{i-1}) \mod  \mathcal J_{i-1}$$
for some $a_i \in \Z_{(p)}^{\times}$.
\end{corollary}
\begin{proof}
Note that $$H^*\left(L_{1(p)} \times S^{2i-1}_{(p)},\Z/p\Z \right) = \Lambda^*(y) \otimes_{\Z} \Lambda^*(y_i)$$ with $|y|=2$ and $|y_i| = 2i-1$ such that the natural map $$ \eta \times \bar\epsilon_{i(p)} \colon  L_{1(p)} \times S^{2i-1}_{(p)} \to \PU(p) \times \PU(p)$$ induces the natural homomorphism from
$$ (\Z/p\Z)[y]/(y^{p}) \otimes_{\Z} \Lambda_{\Z/p\Z}^*(y_1,y_2,\dots,y_{p-1}) \otimes (\Z/p\Z)[y]/(y^{p}) \otimes_{\Z} \Lambda_{\Z/p\Z}^*(y_1,y_2,\dots,y_{p-1})$$
to $\Lambda^*(y) \otimes_{\Z} \Lambda^*(y_i)$ which sends $y \otimes 1$ to $y \otimes 1$, $1 \otimes y_i$ to $1 \otimes y_i$ and the other generators to zero. The kernel of this homomorphism is precisely the ideal $\mathcal J_i$.
Now, Lemma \ref{lem:crucial} implies that the composition
$$L_{1(p)} \wedge S^{2i-1}_{(p)} \to \PU(p)_{(p)} \wedge \PU(p)_{(p)} \stackrel{c}{\to} \SU(p)_{(p)} \stackrel{\lambda_{i+1}}{\to} S_{(p)}^{2i+1}$$ is homotopic to $a_i(q_1 \wedge 1_{S^{2i-1}})_{(p)}$ for some $a_i \in \Z_{(p)}^{\times}$. Since the map $q_1 \wedge 1_{S^{2i-1}} \colon  L_1 \wedge S^{2i-1} \to S^{2i+1}$ sends the generator $x_{i+1}$ of the cohomology of $S^{2i+1}$ to $y \otimes y_i \in \Lambda^*(y) \otimes_{\Z} \Lambda^*(y_i)$, this implies the claim.
\end{proof}

\section{Solvability of equations} 
\label{solv}

\subsection{Hyperlinear groups and related classes of groups}
\label{hyperlinear}
The unitary group $\U(n)$ is equipped with a natural metric that arises from the normalized Frobenius norm, i.e.,
$$d(u,v) = \frac1{n^{1/2}} \left( \sum_{i,j=1}^n |u_{ij} - v_{ij}|^2 \right)^{1/2}.$$
Informally speaking, a group $G$ is said to be hyperlinear if its multiplication table can be modelled locally (that means on finite subsets of the group) by unitary matrices up to small mistakes measured in the normalized Frobenius norm. More precisely:

\begin{definition}
A group $G$ is called {\it hyperlinear} if for all finite subsets $F \subset G$ and all $\varepsilon>0$, there exists $n \in \N$ and a map $\varphi \colon  G \to \U(n)$, such that
\begin{enumerate}
\item $d(\varphi(gh),\varphi(g)\varphi(h))< \varepsilon$ for all $g,h \in F$, and
\item $d(\varphi(g),1_n) > 1$, for all $g \in F \setminus \{e\}$.
\end{enumerate}
\end{definition}
There are variations on this definition but they are all equivalent. A detailed discussion of the class of hyperlinear groups can be found in \cite{MR2460675}. If in the above definition the unitary groups with their metrics are replaced by symmetric groups ${\rm Sym}(n)$ with the normalized Hamming metrics, then one obtains the definition of the class of sofic groups. This important class of groups was introduced by Gromov \cite{gromov} and Weiss \cite{weiss} in order to study certain problems in ergodic theory. Since the inclusion ${\rm Sym}(n) \subset \U(n)$ is compatible enough with the metrics, every sofic group is automatically hyperlinear -- see  \cite{MR2460675} for more details. It is not known if there are non-sofic groups.

We denote the set of prime numbers by $\mathbb P$. We will need the following easy proposition. 

\begin{proposition} \label{fixp}
Let $G$ be a countable hyperlinear group. Then $G$ is a subgroup of a quotient of $\prod_{p \in \mathbb P} \SU(p)$.
\end{proposition}
\begin{proof}
Let $(g_i)_{i \in \N}$ be an enumeration of $G$. Let $n_k$ be some integer such that the definition of hyperlinearity is satisfied for the finite set $\{g_1,\dots,g_k\}$ with $\varepsilon=1/k$. Let $\varphi_k \colon  G \to \U(n_k)$ the the corresponding map. Without loss of generality, we may assume that $\lim_{k \to \infty} n_k = \infty$. Indeed, the natural diagonal embedding $\U(n) \subset \U(nm)$ is isometric with respect to the normalized Frobenius metric, so that we can replace $n_k$ by $k n_k$ if necessary. Using the natural embedding $\U(n) \subset \SU(n+1)$, we may now assume without loss of generality that the image of $\varphi_k$ lies in $\SU(n_k)$.
Similarly, replacing $n_k$ again by a suitable number of the form $mn_k +1$, we may assume that $n_k$ is a prime number -- using Dirichlet's theorem.
Now, consider
$$\mathcal N := \left\{ (u_p)_p \in \prod_{p \in \mathbb P} \SU(p) \mid \lim_{k \to \infty} d(1_{n_k}, u_{n_k})=0 \right\}.$$
It is easy to see that ${\mathcal N} \subset \prod_{p \in \mathbb P} \SU(p)$ is a normal subgroup and $\varphi = \prod_{k \in \N} \varphi_k$ defines a injective homomorphism from $G$ to the quotient of $\prod_{p \in \mathbb P} \SU(p)$ by $\mathcal N$. This proves the claim.
\end{proof}

\begin{remark} \label{hyperl}
It is also true that any subgroup of a quotient of $\prod_{n \in \N} \U(n)$ (or $\prod_{p \in \mathbb P} \SU(p)$ for that matter) is hyperlinear. This follows from results in \cite{stolzthom}, based on ideas from work of Nikolov-Segal, see \cite{MR2995181}. It is not known if there are any groups that are not hyperlinear -- essentially all groups that are known to by hyperlinear, are also known to be sofic.
\end{remark}

\subsection{Proofs of the main results}

\label{sec2}

We can now prove the main theorems, i.e., Theorem \ref{mainthm} and Theorem \ref{sup}. Let us first study the commutator map again.
It follows from Corollary \ref{maincor} that 
\begin{equation} \label{crucial}
c^*(x_2\cdots x_{p}) = a \cdot (y^{p-1} \otimes y_1\cdots y_{p-1}) \mod \mathcal J \end{equation} for some $a \in \Z_{(p)}^{\times}$, where $\mathcal J$ is the ideal defined as
$$\mathcal J:=\sum_{E \subsetneq \{1,\dots,p-1\}} \prod_{i \in E} (y \otimes y_i) \cdot \prod_{i \not \in E} \mathcal J_i^{\geq (2,2i-1)},$$
where ${\mathcal J}_i^{\geq (2,2i-1)}$ denotes the subspace of ${\mathcal J}_i$, which is of bi-degree $(2,2i-1)$ or more. Here, we use the natural bi-grading of the tensor product $H^*(\PU(p),\Z/p\Z) \otimes_{\Z} H^*(\PU(p),\Z/p\Z)$. 
It is easy to see that $y^{p-1} \otimes y_1\cdots y_{p-1} \not \in \mathcal J$. Indeed, assume that $y^{p-1} \otimes y_1\cdots y_{p-1}  \in \mathcal J$. The element $y^{p-1} \otimes y_1\cdots y_{p-1}$ is of bi-degree $(2(p-1), 1+3+\dots+2p-1)$ so that all contributions of $\mathcal J^{\geq (2,2i-1)}_i$ must of of minimal possible degree $(2,2i-1)$. Note that $y \otimes y_i \not \in \mathcal J_i$ by construction. This implies that for any $E \subset \{1,\dots,p-1\}$, a product of $p-1$ factors of the form $y \otimes y_i$ for $i \in E$ and otherwise in $\mathcal J_i^{(2,2i-1)}$ cannot yield a summand $y^{p-1}\otimes y_1\dots y_{p-1}$ unless $E=\{1,\dots,p-1\}$, but this is forbidden.

This proves as a first step that any continuous map which is homotopic to the commutator map $c \colon \PU(p) \times \PU(p) \to \SU(p)$ must be surjective. Indeed, the previous computation shows that its effect on the fundamental class $x_2\cdots x_{p} \in H^{p^2-1}(\SU(p),\Z/p\Z)$ is non-trivial on cohomology -- and this happens only if the map (and any map homotopic to it) is surjective.

We now attempt to extend this result to other words in ${\mathbf F}_2$. Later, we will also allow to vary $p$ and will see that our approach works for all elements which do not lie in $[{\mathbf F}_2,[{\mathbf F}_2,{\mathbf F}_2]]$.
First, we need some preparations. Note that ${\mathbf F}_2^{(1)} := [{\mathbf F}_2,{\mathbf F}_2]$ is a free group with basis $\{[x_1^n,x_2^m] \mid nm \neq 0 \}$, see Proposition 4 in Chapter I, \S1.3 of \cite{trees}.

\begin{proposition}
Let $w = [x_1^{n},x_2^{m}]$. Then
$$w^* \colon  H^*(\SU(p),\Z/p\Z) \to H^*(\PU(p),\Z/p\Z) \otimes H^*(\PU(p),\Z/p\Z)$$ 
satisfies $w^*(x_i) =nm \cdot a_i \cdot (y \otimes y_{i-1}) \mod  \mathcal J_{i-1}$ for some $a_i \in \Z_{(p)}^{\times}$ independent of $n,m$. More generally, if $w = \prod_{k=1}^s [x_1^{n_k},x_2^{m_k}]^{l_k}$, then
$$w^*(x_i) = \left(\sum_{k=1}^s n_km_kl_k \right) \cdot a_i \cdot (y \otimes y_{i-1}) \mod  \mathcal J_{i-1}.$$
On the other side, if $w \in {\mathbf F}_2^{(2)}$, the second derived subgroup, then we have $w^*=0$.
\end{proposition}
\begin{proof} If
$w_1 \colon  \PU(p) \times \PU(p) \to \SU(p)$ and $w_2 \colon  \PU(p) \times \PU(p) \to \SU(p)$ are word maps, for $w_1,w_2 \in [\F_2,\F_2]$, so that the associated word map can be factored through $\SU(p)$, then 
$$w_1w_2 \colon  \PU(p) \times \PU(p) \to \SU(p)$$ is equal
to $$m_{\SU(p)} \circ (w_1 \times w_2) \circ \Delta_{\PU(p) \times \PU(p)} \colon  \PU(p) \times \PU(p) \to \SU(p).$$
Hence, the effect on cohomology can be computed explicitly as follows:
\begin{eqnarray*}
x_i &\mapsto& x_i \otimes 1 + 1 \otimes x_i \\
&\mapsto& w_1^*(x_i) \otimes 1 + 1 \otimes w_2^*(x_i) \\
&\mapsto& w_1^*(x_i) + w_2^*(x_i).
\end{eqnarray*}
The results follow directly from Lemma \ref{coprod} and Corollary \ref{maincor}.
This finishes the proof of the proposition.
\end{proof}

Consider now the 3-dimensional Heisenberg group 
$${\mathbf H}_3(\Z) := \left\{ \left( \begin{matrix} 1 & a & b \\ 0 & 1 & c \\ 0 & 0 & 1 \end{matrix} \right) \mid a,b,c \in \Z \right \}.$$ Sending $x_1$ to the matrix with $a=1,b=c=0$ and $x_2$ to the matrix with $a=b=0, c=1$, we get that $w = \prod_{k=1}^l [x_1^{n_k},x_2^{m_k}]^{l_k}$ is sent to the matrix with $a=c=0$ and $b= \sum_{k=1}^l n_km_kl_k$. It is well known that ${\mathbf H}_3(\Z) = {\mathbf F}_2/[[{\mathbf F}_2,{\mathbf F}_2],{\mathbf F}_2]$.
Coming back to the proof of Theorem \ref{sup}, we see that we will succeed with our strategy if $w$ can be mapped non-trivially to the central quotient of ${\mathbf H}_3(\Z)$ by $p\Z$.
\begin{proof}[Proof of Theorem \ref{sup}]
The result is clear if $\varepsilon(w) \not \in [{\mathbf F}_2,{\mathbf F}_2]$, so that we may assume that $\varepsilon(w) \in [{\mathbf F}_2,{\mathbf F}_2]$. Now, the assumption on $\varepsilon(w)$ implies that it can be mapped non-trivially to the canonical central extension of $\Z^2$ by $\Z/p\Z$, which is just the central quotient of ${\mathbf H}_3(\Z)/p\Z$. This happens if any only if $\varepsilon(w) \not \in [{\mathbf F}_2,{\mathbf F}_2]^p [[{\mathbf F}_2,{\mathbf F}_2],{\mathbf F}_2]$. Hence, by our main result above the induced word map $w \colon  \PU(p) \times \PU(p) \to \SU(p)$ is is non-trivial on the fundamental class $x_1\dots x_n \in H^{p^2-1}(\SU(p),\Z/p\Z)$, and hence the word map must be surjective. Since $\varepsilon(w) \not\in {\mathbf F}_2^{(1)}$, any lift to $\SU(p)$ of a preimage of the neutral element solves the equation.
\end{proof}

\begin{proof}[Proof of Theorem \ref{mainthm}]
This is a straightforward consequence of Theorem \ref{sup} and Proposition \ref{fixp}. The claim about finite groups follows from Mal'cev's theorem \cite{malcev}, stating that finitely generated linear groups are residually finite.
\end{proof}

\subsection{Related results and low-dimensional cases}

Let us finish by mentioning a few low-dimensional results which go beyond the second step of the lower central series. So far, we are unable to exploit the mechanisms behind these examples in order to get satisfactory results for all hyperlinear groups. However, we would also like to mention some further directions and possible extensions of the techniques used in this paper

Recall, we denote the commutator by $[x,y]:= xyx^{-1}y^{-1}$.
The iterated commutators $c_n \in {\mathbf F}_n = \langle x_1,\dots,x_n \rangle$ are defined by induction
$c_n = [x_n,c_{n-1}]$ and $c_1 = x_1$.
The first result that goes beyond the second step in the lower central series is the following result by Porter.

\begin{theorem}[Porter, see \cite{MR0169244}]
The map $c_3^{{\rm \SU(2)}} \colon  {\rm SU}(2)^{\times 3} \to {\rm SU}(2)$ is not null-homotopic.
\end{theorem}

In order to treat $c_n$ for $n \geq 4$, we need to use some more sophisticated results from algebraic topology related to homotopy nilpotence results. This was done by Rao in \cite{MR1217066}, showing also that ${\rm Spin}(n)$ is not homotopy nilpotent for $n \geq 7$.

\begin{theorem}[Rao, \cite{MR1217066}]
The map $c_n^{{\rm PU(2)}} \colon  {\rm PU}(2)^{\times n} \to {\rm PU}(2)$ is not null-homotopic for any $n \in \N$.
\end{theorem}

As before, we obtain results concerning solvability of equations.

\begin{corollary}
Let $G$ be any subgroup of ${\rm SU(2)}$ and let $w \in G \ast {\mathbf F}_n$ be such that $\varepsilon(w)=c_n$ for some $n \geq 2$. Then, $w(x_1,\dots,x_n)=1$ can be solved in some group containing $G$.
\end{corollary}
\begin{proof}
Since ${\rm SU}(2)=S^3$, every non-surjective map is null-homotopic. Thus, using the same arguments as before, we can conclude that for every $w \in {\rm SU}(2) \ast {\mathbf F}_n$ with content $c_n$, the induced word map $w \colon  {\rm \SU(2)}^{\times n} \to {\rm SU}(2)$ is surjective. 
\end{proof}

Finally, we want to mention some questions that appear naturally at this interface between homotopy theory and the study of word maps.
Given a topological group, it seems natural to study the group of words modulo those which are null-homotopic. Let $G$ be a compact Lie group, set $$N_{n,G} := \left\{w \in {\mathbf F}_n \mid w \colon G^n \to G \mbox{ is
homotopically trivial} \right\}$$ and define ${\mathcal H}_{n,G} := {\mathbf F}_n/N_{n,G}$.

\begin{question}
Can we compute ${\mathcal H}_{2,{\rm SU}(n)}$ ? 
 \end{question}

See \cite{MR0106465, MR1233412} for partial information about  $\mathcal H_{n,G}$ in particular cases.
 In this direction, the following result is implied by results on p.464 in \cite{MR516508}.
\begin{theorem}[Whitehead, see \cite{MR516508}]
Let $G$ be a connected and simply connected compact Lie group. Then, ${\mathcal H}_G$ is $k$-step nilpotent for some $k \leq 2 \cdot \dim(G)$.
\end{theorem}
\begin{proof} We denote the degree of nilpotency of a group $\Gamma$ by ${\rm nil}(\Gamma)$.
Whitehead showed that the homotopy set $[X,G]$ is a group and $${\rm nil}([X,G]) \leq \dim(X).$$ For $X=G\times G$, we obtain ${\rm nil}([G \times G,G]) \leq \dim(G \times G) = 2 \cdot \dim(G)$. Now, the subgroup generated by the coordinate projections is precisely ${\mathcal H}_G$. This proves the claim.
\end{proof}

Let $G$ be a topological group, e.g., a compact Lie group. We call $w \in {\mathbf F}_n$ {\it homotopically surjective} with respect to $G$ if every map in the homotopy class of $w \colon  G^{\times n} \to G$ is surjective.

\begin{question}
Let $w \in \F_n \setminus \{1\}$. Is $w \colon \PU(n) \times \PU(n) \to \PU(n)$ homotopically surjective for large $n$?
\end{question}

\section*{Acknowledgments}

The work of the first author was supported by the Russian Foundation for Basic Research, project no. 15-01-05823. He thanks TU Dresden for its hospitality during a visit in January 2015. He thanks this institution for its hospitality. 

The second author wants to thank Tilman Bauer for some interesting comments. 
Part of this paper was written, when he visited Institute Henri Poincar\'e for the trimester about random walks and asymptotic geometry of groups in spring 2014.  This research was supported by ERC Starting Grant No.\ 277728 and the MPI-MIS in Leipzig.

We thank the unknown referee for many useful comments that improved the exposition.

\end{document}